\documentclass{amsart}

\usepackage{amsmath,amsfonts,amssymb,amsthm,amscd,comment,euscript}
\usepackage{stmaryrd}
\usepackage[all]{xy}
\usepackage{graphicx}
\usepackage{enumerate} %allows to change enumeration items easily
\usepackage[colorlinks=true]{hyperref} %to generate pdf with links
\usepackage[usenames,dvipsnames]{xcolor}

\usepackage{hyperref} %to generate pdf with links

\swapnumbers
\theoremstyle{plain}
\newtheorem{lem}{Lemma}[section]
\newtheorem{cor}[lem]{Corollary}
\newtheorem{prop}[lem]{Proposition}
\newtheorem{thm}[lem]{Theorem}
\newtheorem*{thmm}{Theorem}

\theoremstyle{definition}
\newtheorem{ex}[lem]{Example}

\newtheorem{dfn}[lem]{Definition}

\newcommand{\Z}{\mathbb{Z}}      % integers
\newcommand{\Q}{\mathbb{Q}}     % rational numbers

\newcommand{\ssI}{\widetilde{I}}   % augmentation ideal in the simply connected case
\newcommand{\sG}{\widetilde{G}}  % simply connected cover of G

\newcommand{\cs}{\mathfrak{c}}        % characteristic map

\newcommand{\gX}{X^\mathrm{gen}}        % versal flag
\newcommand{\gU}{U^\mathrm{gen}}        % versal torsor
\newcommand{\sH}{\mathcal{H}}              % associated cohomology sheaf
\newcommand{\Os}{\mathcal{O}}                  % structure sheaf

\newcommand{\Spec}{\operatorname{Spec}}   % Spectrum
\newcommand{\Dec}{\operatorname{Dec}}      % Decomposable group
\newcommand{\SDec}{\operatorname{SDec}}    % Semi-decomposable group
\newcommand{\Pic}{\operatorname{Pic}}      % the Picard group
\newcommand{\Inv}{\operatorname{Inv}}      % Groups of invariants
\newcommand{\Nrd}{\operatorname{Nrd}}   % Reduced norm
\newcommand{\Sym}{\operatorname{Sym}}  % Symmetric algebra
\newcommand{\CH}{\operatorname{CH}}             % Chow group
\newcommand{\im}{\operatorname{im}}            % image
\newcommand{\charac}{\operatorname{char}}   % Characteristic of a field

\newcommand{\norm}{\mathrm{norm}}   % normalized invariants
\newcommand{\dec}{\mathrm{dec}}        % decomposable invariants
\newcommand{\sdec}{\mathrm{sdec}}     % semi-decomposable invariants
\newcommand{\ind}{\mathrm{ind}}         % indecomposable invariants
\newcommand{\tors}{\mathrm{tors}}       % torsion part

\newcommand{\tCH}{\scriptscriptstyle{\rm CH}}   % index Chow group
\newcommand{\tK}{\scriptscriptstyle{\rm K_0}}    % index K_0

\newcommand{\QZtwo}{\Q/\Z(2)}
\newcommand{\QZd}{\Q/\Z(d)}

\newcommand{\gPGSp}{\operatorname{\mathbf{PGSp}}}
\newcommand{\gSL}{\operatorname{\mathbf{SL}}}
\newcommand{\gPGL}{\operatorname{\mathbf{PGL}}}
\newcommand{\gHSpin}{\operatorname{\mathbf{HSpin}}}
\newcommand{\gPGO}{\operatorname{\mathbf{PGO}}}
\newcommand{\gSO}{\operatorname{\mathbf{SO}}}
\newcommand{\gSpin}{\operatorname{\mathbf{Spin}}}
\newcommand{\Gm}{\mathbb{G}_m}         % multiplicative group
\newcommand{\gmu}{\boldsymbol{\mu}}         % roots of unity

\title{Invariants of degree 3 and torsion in the Chow group of a versal flag}

\author{Alexander Merkurjev} 
\author{Alexander Neshitov}
\author{Kirill Zainoulline}

\address{Alexander Merkurjev, University of California at Los Angeles}
\email{merkurev@math.ucla.edu}

\address{Alexander Neshitov, Steklov Mathematical Institute,  St.Petersburg / University of Ottawa}
\email{neshitov@yandex.ru}

\address{Kirill Zainoulline, University of Ottawa}
\email{kirill@uottawa.ca}

\thanks{The second author was supported by the Trillium Foundation (Ontario) and the RFBR grant 12-01-33057. The last author was supported by the NSERC Discovery grant  385795-2010 and the Early Researcher Award (Ontario).}

\subjclass[2010]{11E72, 14M17, 14F43}

\keywords{cohomological invariant, linear algebraic group, torsor, Galois cohomology, Chow group}

\begin{document}

\begin{abstract}
We prove that the group of normalized cohomological invariants of
degree 3 modulo the subgroup of semidecomposable invariants of a
semisimple split linear algebraic group $G$ is isomorphic to the torsion
part of the Chow group of codimension 2 cycles of the respective
versal $G$-flag.
In particular, if $G$ is simple, we show that this factor group is
isomorphic to the group of indecomposable invariants of $G$.
As an application, we construct nontrivial cohomological classes for
indecomposable central simple algebras.
\end{abstract}

\maketitle

%%%%%%%%%%%%%%%%%%%%%%%%%%%%%%%%%%
%%  Document starts here

\section{Introduction}

Let $G$ be a split semisimple linear algebraic group over a field $F$.
The purpose of the present paper is to relate together three different
topics: the {\em geometry} of twisted $G$-flag varieties, the theory
of {\em cohomological
invariants} of $G$ and the {\em representation theory} of $G$.

\medskip

As for the first,
let $U/G$ be a {\em classifying space} of $G$ in the sense of Totaro, that is $U$ is an open
$G$-invariant subset in some representation of $G$ with $U(F)\neq
\emptyset$ and $U\to U/G$ is a $G$-torsor. Consider the generic fiber $\gU$ of $U$ over $U/G$. It is a $G$-torsor over the quotient
field $K$ of $U/G$ called the {\em versal} $G$-torsor \cite[Ch.~I,~\S
5]{GMS}. We denote by $\gX$ the respective flag variety $\gU/B$ over
$K$, where $B$ is a Borel subgroup of $G$, and call it the {\em
  versal} flag. The variety $\gX$ can be viewed as the `most
twisted' form of the `most complicated' $G$-flag variety and, hence,
is the most natural object to study.
In particular, understanding its geometry via studying the {\em Chow group} $\CH(\gX)$ of
algebraic cycles modulo the rational equivalence relation, leads to
understanding the geometry of all other $G$-flag varieties.

Recall that the group $\CH(X)$ of a twisted flag variety $X$ has been
a subject of intensive investigations for decades: started with
fundamental results by Grothendieck, Demazure,
Berstein-Gelfand-Gelfand in 70's describing its {\em free part},
inspired by its close connections to the motivic cohomology discovered in
90's by Voevodsky and 
numerous results by Karpenko, Peyre, and
many others (including the authors of the present paper) trying to
estimate its {\em torsion part}. 

\medskip

Our second ingredient, the theory of cohomological invariants, has
been mainly inspired by the works of J.-P.~Serre and M.~Rost.
Given a field extension $L/F$ and a positive integer $d$ we consider the Galois cohomology group $H^{d+1}(L,\QZd)$ denoted by $H^{d+1}(L,d)$.
Following~\cite[Ch.~II,~\S 1]{GMS} a degree $d$ {\em cohomological invariant} is a natural transformation of functors
\[
a\colon H^1(\,\text{---}\,,G)\to H^d(\,\text{---}\,,d-1)
\]
on the category of field extensions over $F$. We denote the group of
degree $d$ invariants by $\Inv^d(G,d-1)$.

Following \cite[\S1]{Merkurjev}
invariant $a$ is called {\em normalized} if it sends trivial torsor to
zero. We denote the subgroup of normalized invariants by $\Inv^d(G,d-1)_\norm$.
Invariant $a$ is called {\em decomposable} if it is given by a
cup-product with an invariant of degree $2$. We denote the subgroup of
decomposable invariants by $\Inv^3(G,2)_\dec$. The factor group
$\Inv^3(G,2)_\norm/\Inv^3(G,2)_\dec$ is denoted by $\Inv^3(G,2)_\ind$
and is called the group of {\em indecomposable} invariants.
This group has been studied by Garibaldi, Kahn, Levine, Rost,
Serre and others in the simply-connected case and is closely related
to the celebrated Rost-Serre invariant. 
In recent work \cite{Merkurjev} it was shown how to
compute it in general using new results on motivic cohomology
obtained in \cite{MerkurjevBG}. In particular, it was computed for all
adjoint split groups in \cite{Merkurjev} and for split simple groups in \cite{BR}.

\medskip

As for the last ingredient, the representation theory of $G$,
it had established itself a long time ago originating from the
theory of Lie algebras in the middle of the last century.
Recall that the classical {\em character map} identifies
the representation ring of $G$ with the subring
$\Z[T^*]^W$ of $W$-invariant elements of the integral group ring
$\Z[T^*]$,
where $W$ is the Weyl group which acts naturally on the group of
characters $T^*$ of a split maximal torus $T$ of $G$, hence, providing
a straightforward link to the classical Invariant theory.

\medskip

We glue all these ingredients together by introducing a new subgroup of
{\em semi-decomposable} invariants $\Inv^3(G,2)_\sdec$ which consists
of invariants $a \in \Inv^3(G,2)_\norm$ such that for every field
extension $L/F$ and a $G$-torsor $Y$ over $L$
\[
a(Y)=\sum_{i\; \text{finite}} \phi_i\cup b_i(Y)\text{ for some }\phi_i\in
L^{\times}\text{ and }b_i\in \Inv^2(G,1)_\norm.
\]
Roughly speaking, it consists of invariants that are `locally decomposable'. 
Observe that by definition $\Inv^3(G,2)_\dec \subseteq
\Inv^3(G,2)_\sdec \subseteq \Inv^3(G,2)_\norm$.

\medskip

Our main result then says that

\begin{thmm} Let $G$ be a split semisimple linear algebraic group
  over a field $F$ and let $\gX$ denote the associated versal flag.
There is a short exact sequence 
\[
0\to \tfrac{\Inv^3(G,2)_\sdec}{\Inv^3(G,2)_\dec} \to
\Inv^3(G,2)_{\operatorname{ind}}\to\CH^2(\gX)_\tors\to 0,
\]
together with a group isomorphism
$\tfrac{\Inv^3(G,2)_\sdec}{\Inv^3(G,2)_\dec}\simeq
\tfrac{c_2((\ssI^W)\cap\Z[T^*])}{c_2(\Z[T^*]^W)}$,
where
$c_2$ is the second {\em Chern class} map
(e.g. see \cite[\S 3c]{Merkurjev}) and
$(\ssI^W)$ denote
the ideal generated by classes of augmented representations of the
simply-connected cover of $G$.

In addition, if $G$ is simple,
then $\Inv^3(G,2)_\sdec= \Inv^3(G,2)_\dec$, so there is an isomorphism 
$\Inv^3(G,2)_\ind \simeq \CH^2(\gX)_\tors$.
\end{thmm}

Observe that if $G$ is not simple, then $\Inv^3(G,2)_\sdec$ does not
necessarily coincide with $\Inv^3(G,2)_\dec$ (see Example~\ref{counterO4}).

\medskip

The nature of our result suggests that it should have applications
in several directions, e.g. for cohomological invariants and algebraic cycles on
twisted flag varieties. In the present paper we discuss only few of them.

\medskip

For instance, since the group $\Inv^3(G,2)_{\ind}$ has been computed for all
simple split groups in \cite{Merkurjev} and \cite{BR}, it immediately
gives computation of the torsion part of $\CH^2(\gX)$, hence,
extending previous results by \cite{Ka98} and
\cite{Peyre}. 
As another straightforward consequence, using the coincidence
$\Inv^3(G,2)_\sdec=\Inv^3(G,2)_\dec$ we construct non-trivial
cohomological classes for indecomposable central simple algebras,
hence, answering questions posed in \cite{GPT09} and \cite{Demba13}.

\medskip

The paper is organized as follows: In section~\ref{direct} we
construct an exact sequence relating the groups of invariants with the
torsion part of the Chow group, hence, proving the first part of the
theorem. In section~\ref{coincidence} we compute this exact sequence
case by case for all simple groups, hence, proving the second part. In
the last section we discuss applications.

\section{Semi-decomposable invariants and the Chow group}\label{direct}

Let $G$ be a split semisimple linear algebraic group over a field
$F$. We fix a split maximal torus $T$ of $G$ and a Borel subgroup $B$
containing $T$. 
Consider the $T$-equivariant structure map $U\to \Spec F=pt$, where $U$ is the open $G$-invariant
subset from the introduction. 

\subsubsection*{Characteristic maps and classes}
By~\cite{EG} the induced pullback on
$T$-equivariant Chow groups $\CH_T(pt) \to \CH_T(U)$ is an isomorphism.
Since $\CH_T(U)\simeq \CH(U/T)\simeq
\CH(U/B)$ and $\CH_T(pt)$ can be identified with the symmetric algebra
$\Sym(T^*)$ of the group of characters of $T$, it gives an isomorphism
\begin{equation}\label{CHisoU}
\cs^{\tCH}\colon \Sym(T^*) \xrightarrow{\simeq} \CH(U/B).
\end{equation}

Similarly, by the homotopy invariance and localization property of the
equivariant $K$-theory~\cite[Theorems~8 and~11]{MerkurjevK} the induced pull-back on
$T$-equivariant $K$-groups gives a surjection
\[
\cs^{\tK}\colon \Z[T^*]\twoheadrightarrow K_0(U/B),
\]
where the integral group ring $\Z[T^*]$ can be identified with $K_T(pt)$ and $K_0(U/B)\simeq K_0(U/T)\simeq K_T(U)$.

\medskip

Let $\tau^i(X)$ denote the $i$-th term of the {\em topological
  filtration} on $K_0$ of a smooth variety $X$ and let $\tau^{i/i+1}$,
$i\ge 0$ denote its $i$-th subsequent quotient. Let $I$ denote the
augmentation ideal of $\Z[T^*]$.

\begin{lem}\label{isoquot}
The map $\cs^{\tK}$ induces isomorphisms on subsequent quotients 
\[
 I^i/I^{i+1} \xrightarrow{\simeq} \tau^{i/i+1}(U/B),\quad \text{ for }0\le i\le 2,
\]
and, its restriction $\cs^{\tK}\colon I^2 \to \tau^2(U/B)$ is surjective.
\end{lem}

\begin{proof}
By \cite[Ex.~15.3.6]{Fu} the Chern class maps induce isomorphisms $c_i \colon \tau^{i/i+1}(X) \xrightarrow{\simeq} \CH^i(X)$ for $0\le i\le 2$. Since the Chern classes commute with pullbacks and
$\Sym^i(T^*)\simeq I^i/I^{i+1}$, the isomorphisms then follow
from~\eqref{CHisoU}.

Finally, since $I/I^2 \simeq \tau^{1/2}(U/B)$, $\cs^{\tK}(x)\in \tau^2(U/B)$ implies
that $x\in I^2$.
\end{proof}

Consider the natural inclusion of the versal flag $\imath\colon
\gX=\gU/B \hookrightarrow U/B$. Since $\imath$ is a limit of open embeddings, by
the {\em localization property} of Chow groups, the induced pullback
gives surjections
\[
\imath^{\tCH}\colon\CH^i(U/B)\twoheadrightarrow \CH^i(\gX).
\]
Moreover, the induced pullback in $K$-theory restricted to $\tau^i$
also gives surjections
\[
\imath^{\tK}\colon \tau^i(U/B) \twoheadrightarrow \tau^i(\gX).
\]
Indeed, by definition $\tau^i(\gX)$ is generated by the classes $[\Os_Z]$ for closed subvarieties $Z$ of $\gX$ with $\operatorname{codim} Z\geqslant i$ and each $[\Os_Z]$ is the pullback of the element $[\Os_{\bar{Z}}]$ in $\tau^i(U/B)$, where $\bar{Z}$ is the closure of $Z$ inside $U/B$.

\medskip

Let $L$ be a splitting field of the versal torsor $\gU$. According to~\cite[Thm.~4.5]{GiZa} composites 
\[
\Sym(T^*) \xrightarrow{\cs^{\tCH}} \CH(U/B) \xrightarrow{\imath^{\tCH}} \CH(\gX) \xrightarrow{res} \CH(\gX_L)\quad \text{ and}
\] 
\[
\Z[T^*] \xrightarrow{\cs^{\tK}} K_0(U/B) \xrightarrow{\imath^{\tK}} K_0(\gX) \xrightarrow{res} K_0(\gX_L)
\] 
give the classical characteristic maps for the Chow groups and for the
$K$-groups respectively (here we identify the rightmost groups with
the Chow group and the $K$-group of the split flag $G/B$
respectively). Restricting the latter to $I^2$ and $\tau^2$ we obtain
the map
\[
\cs\colon I^2 \xrightarrow{\cs^{\tK}} \tau^2(U/B) \xrightarrow{\imath^{\tK}} \tau^2(\gX) \xrightarrow{res} \tau^2(\gX_L)=\tau^2(G/B).
\]

From this point on, we denote by $\cs^{\tCH}$, $\imath^{\tCH}$ and by
$\cs^{\tK}$, $\imath^{\tK}$ the
respective restrictions to $\Sym^2$, $\CH^2$ and $I^2$, $\tau^2$.

Let $\Lambda$ be the weight lattice. Consider the integral group ring
$\Z[\Lambda]$.
Let $\ssI$ denote its augmentation ideal. The Weyl group $W$ acts
naturally on $\Z[\Lambda]$. Let $(\ssI^W)$ denote the
ideal generated by $W$-invariant elements in $\ssI$.

\begin{lem}\label{l2}
The kernel of the composite $I^2\stackrel{\cs^{\tK}}\to \tau^2(U/B)\stackrel{\imath^{\tK}}\to \tau^2(\gX)$ is $(\ssI^W)\cap I^2$.
\end{lem}

\begin{proof}
By the results of Panin~\cite{Panin}, $K_0(\gX)$ is the direct sum of $K_0(A_i)$ for some central simple algebras $A_i$ over $K$. So $K_0(\gX)$ is a free abelian group and, hence, the restriction $\tau^2(\gX)\to \tau^2(\gX_L)$ is injective. Therefore, $\ker (\imath^{\tK}\circ\cs^{\tK})$ coincides with the kernel of the characteristic map $\cs \colon I^2\to \tau^2(G/B)$. Since $\cs$ factors as $I^2\hookrightarrow \ssI^2 \to \tau^2(G/B)$ and the kernel of the second map is $(\ssI^W)\cap \ssI^2$ by the theorem of Steinberg \cite{St},  we get $\ker \cs= (\ssI^W) \cap I^2$ (here we used that $\Z[T^*]\cap \ssI^i=I^i$).
\end{proof}

Consider the second Chern class map $c_2\colon \tau^2(U/B)\to\CH^2(U/B)$.

\begin{lem}\label{l1}
We have $c_2(\ker \imath^{\tK})=\ker \imath^{\tCH}$.
\end{lem}

\begin{proof}
Consider the diagram
\[
\xymatrix{
\tau^3(U/B)\ar[r]\ar[d]_{\imath^{\tK}|_{\tau^3}} & \tau^2(U/B)\ar[r]^-{c_2}\ar[d]^{\imath^{\tK}} & \CH^2(U/B)\ar[r]\ar[d]^{\imath^{\tCH}} & 0\\
\tau^3(\gX)\ar[r] & \tau^2(\gX)\ar[r]^-{c_2} & \CH^2(\gX)\ar[r] & 0\\
}
\]
Its vertical maps are surjective and the rows are exact by \cite[Ex.~15.3.6]{Fu}. The result then follows by the diagram chase.
\end{proof}

Consider the composite $\cs_2\colon I^2\xrightarrow{\cs^{\tK}} \tau^2(U/B)\xrightarrow{c_2}\CH^2(U/B)$. Observe that it coincides with the Chern class map defined in~\cite[\S 3c]{Merkurjev}.

\begin{lem}\label{kernel}
We have $\ker \imath^{\tCH}=\cs_2((\ssI^W)\cap I^2)$.
\end{lem}

\begin{proof}
Since $\cs^{\tK}$ is surjective by lemma~\ref{isoquot}, we have by lemma~\ref{l2} and~\ref{l1}
\[
\cs_2((\ssI^W)\cap I^2)=\cs_2(\ker(\imath^{\tK}\circ\cs^{\tK}))=c_2(\ker \imath^{\tK})=\ker \imath^{\tCH}. \qedhere
\]
\end{proof}

Following~\cite{Merkurjev} we denote 
\[
\Dec(G):=(\cs^{\tCH})^{-1}\circ \cs_2(\Z[T^*]^W).
\]
And we set
\[
\SDec(G):=(\cs^{\tCH})^{-1}\circ \cs_2((\ssI^W)\cap\Z[T^*]).
\]
Since the action of $W$ on $\Lambda$ is essential, i.e. $\Lambda^W=0$,
we have $(\ssI^W)\subseteq\ssI^2$. Therefore, for any $x\in(\ssI^W)$
we have $x\equiv x' \text{ mod }\ssI^3$ and, hence,
$\cs_2(x)=\cs_2(x')$ for some $x'\in\Z[\Lambda]^W$, where
$\Z[\Lambda]^W$ is the subring of $W$-invariants. So there are inclusions 
\begin{equation}\label{modulo3}
\Dec(G)\subseteq\SDec(G)\subseteq \Sym^2(T^*)^W.
\end{equation}

\begin{lem}
We have $\CH^2(\gX)\simeq \Sym^2(T^*)/\SDec(G)$.
\end{lem}

\begin{proof}
By \eqref{CHisoU} and lemma~\ref{kernel} we have 
\[
\CH^2(\gX)\simeq \CH^2(U/B)/\cs_2((\ssI^W)\cap I^2) \simeq \Sym^2(T^*)/\SDec(G). \qedhere
\]
\end{proof}

\begin{cor}\label{chowtwo} 
We have $\CH^2(\gX)_\tors\simeq \Sym^2(T^*)^W/\SDec(G)$.
\end{cor}

\begin{proof}
By the lemma it remains to show that
\[
(\Sym^2(T^*)/\SDec(G))_\tors=\Sym^2(T^*)^W/\SDec(G).
\] 
Indeed, suppose that $x\in \Sym^2(T^*)$ and $nx\in \SDec(G)$. Then $nx$ lies in $\Sym^2(T^*)^W$ by~\eqref{modulo3}. So for every $w\in W$ we have $n(wx-x)=0$. Since $\Sym^2(T^*)$ has no torsion, $x\in \Sym^2(T^*)^W$. Conversely, let $x\in \Sym^2(T^*)^W$. Since the second Chern class map $c_2\colon I^2\to \Sym^2(T^*)$ is surjective, there is a preimage $y\in I^2$ of $x$. Take $y'=\sum_{w\in W}w\cdot y\in \Z[T^*]^W\subseteq(\ssI^W)\cap\Z[T^*]$. Since $c_2$ is $W$-equivariant and coincides with the composite $(\cs^{\tCH})^{-1}\circ \cs_2$, we get $(\cs^{\tCH})^{-1}\circ \cs_2(y')=|W|\cdot x\in\SDec(G)$.
\end{proof}

\subsubsection*{Cohomological Invariants}

For a smooth $F$-scheme $X$ let $\sH^3(2)$ denote the Zariski sheaf on
$X$ associated to a presheaf $W\mapsto
H^3_{\text{\'et}}(W,\QZtwo)$. The Bloch-Ogus-Gabber
theorem~(see \cite{CTHK} and \cite{GrSu}) implies that its group of global sections $H^0_{\text{Zar}}(X,\sH^3(2))$ is a subgroup in $H^3(F(X),2)$. 

\medskip

Consider the versal $G$-torsor $\gU$ over the quotient field $K$ of
the classifying space $U/G$.
By \cite[Thm.~A]{BM} the map $\Theta\colon \Inv^3(G,2)\to H^3(K,2)$ defined by $\Theta(a):=a(\gU)$ gives an inclusion  
\[
\Inv^3(G,2)\hookrightarrow H^0_{\text{Zar}}(U/G,\sH^3(2)).
\]

\begin{lem}\label{theta}
We have $a(\gU)\in\ker[H^3(K,2)\to H^3(K(\gX),2)]$ for any $a\in \Inv^3(G,2)_\norm$.
\end{lem}

\begin{proof}
Consider the composite $q\colon\Spec K(\gU)\to \gU \to U/G$. Observe that the pullback $q^*$ factors as
\[
q^*\colon H^0_{\text{Zar}}(U/G,\sH^3(2))\to H^3(K(\gX),2)\to H^3(K(\gU),2).
\]
Since $\gU\to \gX$ is a $B$-torsor, $K(\gU)$ is purely transcendental over $K(\gX)$, so the last map of the composite is injective. Since the $\gU$ becomes trivial over $K(\gU)$, we have $q^*(a(\gU))=a(\gU\times_K K(\gU))=0$. Therefore, $a(\gU)\in\ker[H^3(K,2)\to H^3(K(\gX),2)]$.
\end{proof}

\begin{lem}\label{trivact}
Let $Y\to\Spec L$ be a $G$-torsor and $X=Y/B$. Let $L^{sep}$ denote the separable closure of $L$, $\Gamma_L$ its Galois group and $X^{sep}=X\times_L L^{sep}$. Then the $\Gamma_L$ action on $\Pic X^{sep}$ is trivial.
\end{lem}

\begin{proof}
It follows by \cite[Prop.~2.2]{MT}.
\end{proof}

\subsubsection*{The Tits map}

Consider a short exact sequence of $F$-group schemes
\[
1\to C\to\sG \stackrel{\pi}\to G\to 1.
\]
Given a character $\chi\in C^*$ of the center and a field extension $L/F$ consider the {\em Tits map}~\cite[\S4,5]{Ti71}
\[
\alpha_{\chi,L}\colon H^1(L,G)\xrightarrow{\partial} H^2(L,C)\xrightarrow{\chi_*} H^2(L,\Gm),
\]
where $\partial$ is the connecting homomorphism (if $C$ is non-smooth, we replace it by $\Gm$ and $G$ by the respective push-out as in~\cite[II, Example 2.1]{GMS}). This gives rise to a cohomological invariant $\beta_\chi$ of degree two
\[
\beta_{\chi}\colon Y\mapsto\alpha_{\chi,L}(Y)\quad\text{ for every }G\text{-torsor }Y\in H^1(L,G).
\]
\cite[Theorem~2.4]{BM} shows that the assignment $\chi\to\beta_\chi$ provides an isomorphism $C^*\to \Inv^2(G,1)$.

For a $G$-torsor $Y$ over $L$ there is an exact sequence studied in~\cite{Merkurjev95},~\cite{Peyre} and~\cite[II, Thm. 8.9]{GMS}:
\[
A^1((Y/B)^{sep},K_2)^{\Gamma}\xrightarrow{\rho}\ker[H^3(L,2)\to H^3(L(Y/B),2)]\xrightarrow{\delta_Y}\CH^2(Y/B).
\]
The multiplication map $L^{sep}\otimes\CH^1(Y/B)^{sep}\to A^1((Y/B)^{sep},K_2)$ is an isomorphism. By lemma~\ref{trivact} we obtain an exact sequence
\begin{equation}\label{sequence}
L\otimes\Lambda\xrightarrow{\rho_Y}\ker[H^3(L,2)\to H^3(L(Y/B),2)]\xrightarrow{\delta_Y}\CH^2(Y/B).
\end{equation}
According to~\cite{Merkurjev95} the map $\rho_Y$ acts as follows: 
\[
\rho_Y(\phi\otimes\lambda)=\phi\cup\beta_{\overline{\lambda}}(Y),\quad\text{ where }\phi\in L^{\times},\; \lambda\in\Lambda\text{ and}
\]
$\overline{\lambda}$ denotes the image of $\lambda$ in $\Lambda/T^*=C^*$.

\medskip

We define the subgroup $\Inv^3(G,2)_\sdec$ of {\em semi-decomposable} invariants as follows:

\begin{dfn}
An invariant $a \in \Inv^3(G,2)_\norm$ is called semi-decomposable, if there is a finite set $b_i\in \Inv^2(G,1)_\norm$ such that for every field extension $L/F$ and a torsor $Y\in H^1(L,G)$ we have
\[
a(Y)=\sum_i \phi_i\cup b_i(Y)\text{ for some }\phi_i\in L^{\times}.
\]
\end{dfn}
Observe that by definition, we have
\[
\Inv^3(G,2)_\dec\subseteq \Inv^3(G,2)_\sdec\subseteq \Inv^3(G,2)_\norm
\]
and $a\in \Inv^3(G,2)_\sdec$ if and only if $a(Y)\in\im(\rho_Y)=\ker(\delta_Y)$ for every torsor $Y$.

\begin{lem}\label{semi} 
We have $a\in\Inv^3(G,2)_\sdec$ if and only if $a(\gU)\in \ker(\delta_{\gU})$.
\end{lem}

\begin{proof}
If $a$ is a semi-decomposable invariant, then $a(\gU)=\sum_{\chi\in C^*} \phi_\chi\cup\beta_\chi(\gU)$ lies in the image of $\rho_{\gU}$, hence, $\delta_{\gU}(a(\gU))=0$. On the other hand, let $a$ be a degree $3$ invariant such that $\delta_{\gU}(a(\gU))=0$ and let $Y$ be a $G$-torsor over a field extension $L/F$. We show that $\delta_Y(a(Y))=0$. 

We may assume that $L$ is infinite (replacing $L$ by $L(t)$ if needed). Choose a rational point $y\in (U/G)_L$ such that $Y$ is isomorphic to the fiber of $U\to U/G$ over $y$. Let $R$ be the completion of the regular local ring $\Os_{(U/G)_L,y}$ and let $\hat K$ be its quotient field. The ring $R$ is a regular local ring with residue field $L$. By the theorem of Grothendieck $Y_R$ is a pullback of $Y$ via the projection $\Spec R\to \Spec L(y)$. Then the $G$-torsors $Y_{\hat K}$ and $\gU_{\hat K}$ over $\hat K$ are isomorphic. We have
\[
\delta_Y(a(Y))_{\hat K}=\delta_{Y_{\hat K}}(a(Y_{\hat K}))=\delta_{\gU_{\hat K}}(a(\gU_{\hat K}))=\delta_{\gU}(a(\gU))_{\hat K}=0.
\]
The restriction $\CH^2(Y/B)\to \CH^2((Y/B)_{\hat K})$ is injective, since it is split by the specialization map with respect to a system of local parameters of $R$. Therefore, $\delta_Y(a(Y))=0$ for every $Y$, hence, $a$ is semi-decomposable.
\end{proof}

Now we are ready to prove the first part of the main theorem:

\begin{thm}\label{exactseq}
The map $\delta_{\gU}$ induces a short exact sequence 
\[
0\longrightarrow \tfrac{\Inv^3(G,2)_\sdec}{\Inv^3(G,2)_\dec} \longrightarrow \Inv^3(G,2)_\ind\xrightarrow{g}\CH^2(\gX)_\tors\longrightarrow 0,
\]
and there is  a group isomorphism
\[
\tfrac{\Inv^3(G,2)_\sdec}{\Inv^3(G,2)_\dec}\simeq \tfrac{c_2((\ssI^W)\cap\Z[T^*])}{c_2(\Z[T^*]^W)}.
\]
\end{thm}

\begin{proof}
Consider the following diagram. The rows are exact sequences given by~\cite[Thm.~1.1]{Kahn} and vertical arrows are pullbacks:
\[
\xymatrix{
0\ar[r] & \CH^2(U/G)\ar[r]\ar[d] & \mathbb{H}^4_{\text{\'et}}(U/G,\Z(2))\ar[r]\ar[d] & H^0_{\text{Zar}}(U/G,\sH^3(2))\ar[r]\ar[d] & 0\\
0\ar[r] & \CH^2(U/B)\ar[r]\ar[d]_{\imath^{\tCH}} & \mathbb{H}^4_{\text{\'et}}(U/B,\Z(2))\ar[r]\ar[d] & H^0_\text{Zar}(U/B,\sH^3(2))\ar[r]\ar[d] & 0\\
0\ar[r] & \CH^2(\gX)\ar[r] & \mathbb{H}^4_{\text{\'et}}(\gX,\Z(2))\ar[r] & H^0_\text{Zar}(\gX,\sH^3(2))\ar[r] & 0
}
\]
Since $F(U/B)=K(\gX)$, lemma~\ref{theta} implies that the composite 
\[
\Inv^3(G,2)_\norm\to H^0_\text{Zar}(U/G,\sH^3(2))\to
H^0_\text{Zar}(U/B,\sH^3(2))
\] 
is zero. By the diagram chase there is a homomorphism 
\[ 
\Inv^3(G,2)_\norm\to\CH^2(U/B)/\CH^2(U/G).
\] 
The map $\gX\to U/B\to U/G$ factors as $\gX\to\Spec K\to U/G$, hence
the composite of pullbacks
$\CH^2(U/G)\to\CH^2(U/B)\xrightarrow{\imath^{\tCH}}\CH^2(\gX)$ coincides with the composite $\CH^2(U/G)\to\CH^2(\Spec K)\to\CH^2(\gX)$ which is
zero. 
This gives a homomorphism $g\colon \Inv^3(G,2)_\norm\to
\CH^2(U/B)/\CH^2(U/G) \to \CH^2(\gX)$ which
by the proof of theorem of B. Kahn (see~\cite[II,\,\S 8, 8.1-8.5]{GMS}) factors through the map $\delta_{\gU}$
of~\eqref{sequence}.
By~\cite[3.9]{Merkurjev} the map $g$ also factors through $\Inv^3(G,2)_\ind\xrightarrow{\simeq}
\tfrac{\Sym^2(T^*)^W}{\Dec(G)}$. So there is a commutative diagram
\begin{equation}\label{maindiagr}
\xymatrix{
\Inv^3(G,2)_\norm\ar[r]^{g}\ar[d] & \CH^2(\gX)_\tors\\
\tfrac{\Sym^2(T^*)^W}{\Dec(G)}\ar[r] &
\tfrac{\Sym^2(T^*)^W}{\SDec(G)}\ar[u]^{\simeq}_{\text{Cor. }\ref{chowtwo}}
}.
\end{equation}
The bottom row of~\eqref{maindiagr} gives a short exact sequence 
\[
0\to \tfrac{\SDec(G)}{\Dec(G)}\to \tfrac{\Sym^2(T^*)^W}{\Dec(G)}\to\CH^2(\gX)_\tors\to 0.
\] 
Lemma~\ref{semi} and composite~\eqref{sequence} give an exact sequence
\[
0\to \Inv^3(G,2)_\sdec\to \Inv^3(G,2)_\norm\xrightarrow{g}\CH^2(\gX)_\tors.
\]
Combining these together and factoring modulo $\Inv^3(G,2)_\dec$ we
obtain an isomorphism
\[
\tfrac{\Inv^3(G,2)_\sdec}{\Inv^3(G,2)_\dec}\cong \tfrac{\SDec(G)}{\Dec(G)}. \qedhere
\]
\end{proof}

\section{Semi-decomposable invariants vs. decomposable invariants}\label{coincidence}

In this section we prove case by case that the groups of decomposable
$\Inv^3(G,2)_\dec$ and semi-decomposable $\Inv^3(G,2)_\sdec$
invariants coincide for all split simple~$G$, hence, proving the second
part of our main theorem. 
More precisely, we show that 
\[\Dec(G)=\cs_2(\Z[T^*]^W)=\cs_2((\ssI^W)\cap\Z[T^*])=\SDec(G)\text{ in }\Sym^2(T^*)^W\] 
(here we denote $(\cs^{\tCH})^{-1}\circ \cs_2$ simply by
$\cs_2$). Observe that in the simply connected case
$\Sym^2(\Lambda)^W=\Z q$,  where $q$ corresponds to the normalized
Killing form from \cite[\S1B]{GaZa}, and $\Dec(G)\subseteq \SDec(G)\subseteq \SDec(\sG)=\Dec(\sG)$.

\begin{ex}\label{counterO4} If $G$ is not simple, then $\Dec(G) \neq
  \SDec(G)$ in general. Indeed, consider a quadratic form $q$ of degree 4 with trivial
  discriminant (it corresponds to a $\gSO_4$-torsor). According to
  \cite[Example~20.3]{GMS} there is an invariant given by $q\mapsto \alpha \cup \beta \cup \gamma$, where
  $\alpha$ is represented by $q$ and $\langle\!\langle
  \beta,\gamma\rangle\!\rangle=\langle \alpha\rangle q$ is the 2-Pfister form.
By definition this invariant is semi-decomposable (this fact was pointed to
us by Vladimir Chernousov). Since it is
non-trivial over an algebraic closure of $F$, it is not decomposable.
\end{ex}

\subsection{\it Adjoint groups of type $A_n$ ($n\ge 1$), $B_n$ ($n\geq 2$), $C_n$ ($n\ge 3$, $4\nmid n$), $D_n$ ($n\ge 5$, $4\nmid n$), $E_6$, $E_7$ and special orthogonal groups of type $D_n$ ($n\ge 4$)}

\ 

\medskip 

For classical adjoint types we have
$\Inv^3(G,2)_\norm=\Inv^3(G,2)_\dec$ by~\cite[\S 4b]{Merkurjev}, so we
immediately obtain $\Inv^3(G,2)_\dec=\Inv^3(G,2)_\sdec$. For
exceptional types by~\cite[p.135]{GMS} and~\cite[\S4b]{Merkurjev}  we
have $\Dec(G)=\Dec(\sG)=6\Z q$ for $E_6$ and $\Dec(G)=\Dec(\sG)=12\Z
q$ for $E_7$. For special orthogonal groups $G=\gSO_{2n}$
by~\cite[\S15]{GMS} we have $\Dec(\gSO_{2n})=\Dec(\gSpin_{2n})=2\Z
q$ (here $\tilde G=\gSpin_{2n}$), hence, $\Dec(G)=\SDec(G)$.

\subsection{\it Non-adjoint groups of type $A_{n-1}$ ($n\ge 4$)}\label{nonadjA}

\ 

\medskip

Let $p$ be a prime integer and $G=\gSL_{p^s}/\gmu_{p^r}$ for some integers $s\geq r>0$. If $p$ is odd, we set $k=\min\{r, s-r\}$ and if $p=2$ we assume that $s\geq r+1$ and set $k=\min\{r, s-r-1\}$.
It is shown in \cite[\S4]{BR} that the group $\Inv^3(G,2)_\ind$ is cyclic of order $p^k$. On the other hand, by~\cite[Example~4.15]{Ka98} if $X$ is the
Severi-Brauer variety of a generic algebra $A^{\mathrm{gen}}$, then $\CH^2(X)_\tors$ is also a cyclic group of order $p^k$. The canonical morphism $\gX \to X$ is an iterated projective bundle, hence, $\CH^2(\gX)_\tors\simeq \CH^2(X)_\tors$ is a cyclic group of order $p^k$. It follows from the exact sequence of theorem~\ref{exactseq} that $\Inv^3(G,2)_\sdec=\Inv^3(G,2)_\dec$.

\medskip

More generally, let $G=\gSL_n/\gmu_m$, where $m\mid n$. Let $p^s$ and $p^r$ be the highest powers of a prime integer $p$ dividing $n$ and $m$ respectively. Consider the canonical homomorphism $H=\gSL_{p^s}/\boldsymbol{\mu}_{p^r}\to G$. We claim that it induces an isomorphism between the $p$-primary component of $\Inv^3(G,2)_\ind$ and the group $\Inv^3(H,2)_\ind$. 

\medskip

Indeed, let $H'=\gSL_{n}/\gmu_{p^r}$. It follows from \cite[Theorem 4.1]{BR} that the natural homomorphism $\Inv^3(H',2)_\ind\to \Inv^3(H,2)_\ind$ is an isomorphism. Thus, it suffices to show that the pull-back map for the canonical surjective homomorphism $H'\to G$ with kernel $\gmu_t$, where $t:=m/p^r$ is relatively prime to $p$, induces an isomorphism between the $p$-primary component of $\Inv^3(G,2)_\ind$ and $\Inv^3(H',2)_\ind$. Let $\Lambda\subset \Lambda'$ be the character groups of maximal tori of $G$ and $H'$ respectively. The factor group $\Lambda'/\Lambda$ is isomorphic to $\gmu_t^*=\Z/t\Z$. Since the functor $\Lambda\mapsto \tfrac{\Sym^2(\Lambda)^W}{\Dec(\Lambda)}$ is quadratic in $\Lambda$, the kernel and the cokernel of the homomorphism
\[
\Inv^3(G,2)_\ind=\tfrac{\Sym^2(\Lambda)^W}{\Dec(\Lambda)}\to \tfrac{\Sym^2(\Lambda')^W}{\Dec(\Lambda')}=\Inv^3(H',2)_\ind
\]
are killed by $t^2$. As $t$ is relatively prime to $p$, the claim follows.

\medskip

Since the $p$-primary component of $\CH(\gX)_\tors$ and the group $\CH(\gX_H)_\tors$ are isomorphic by \cite[Prop.~1.3]{Ka98} (here $\gX_H$ denotes the versal flag for $H$), we obtain that $\Inv^3(G,2)_\ind\simeq \CH(\gX)_\tors$ and, therefore, by the exact sequence of theorem~\ref{exactseq} $\Inv^3(G,2)_\sdec=\Inv^3(G,2)_\dec$.

\subsection{\it Adjoint groups of type $C_{4m}$ $(m\geq 1)$}\label{Cn}

\ 

\medskip

By~\cite[\S4b]{Merkurjev} we have $\Sym^2(T^*)^W=\Z q$ and $\Dec(G)=\cs_2(\Z [T^*]^W)=2\Z q$. We want to show that $\cs_2(x)\in 2\Z q$ for every element $x\in(\ssI^W)\cap\Z[T^*]$.

\medskip

Given a weight $\chi\in \Lambda$ we denote by $W(\chi)$ its $W$-orbit and we define $\widehat{e^{\chi}}:=\sum_{\lambda\in W(\chi)}(1-e^{-\lambda})$. By definition, the ideal $(\ssI^W)$ is generated by elements $\{\widehat{e^{\omega_i}}\}_{i=1..4m}$ corresponding to the fundamental weights $\omega_i$. An element $x$ can be written as 
\begin{equation}\label{alphaid}
x=\sum_{i=1}^{4m}n_i\widehat{e^{\omega_i}}+\delta_i\widehat{e^{\omega_i}},\quad\text{ where }n_i\in \Z\text{ and }\delta_i\in\ssI.
\end{equation}

Similar to \cite[\S3]{Za} consider a ring homomorphism $f\colon\Z[\Lambda]\to\Z[\Lambda/T^*]$ induced by taking the quotient $\Lambda\to\Lambda/T^*=C^*$. We have $\Lambda/T^*\simeq \Z/2\Z$ and $\Z[\Lambda/T^*]=\Z[y]/(y^2-2y)$, where $y=f(e^{\omega_1}-1)$. Observe that $C^*$ is $W$-invariant.

\medskip

By definition, $f(I)=0$, so $f(x)=0$. Since $\omega_i\in T^*$ for all even $i$, $f(\widehat{e^{\omega_i}})=y$ for all odd $i$ and $f(\delta_i)\in f(\ssI)=(y)$, we get 
\[
0=f(x)=\sum_{i \text{ is odd}}n_i d_i y+ m_id_i y^2=(\sum_{i \text{ is odd}}n_i + 2 m_i)d_iy,
\]
where $m_i\in \Z$ and $d_i=2^i{4m \choose i}$ is the cardinality of $W(\omega_i)$, which implies that $(\sum_{i \text{ is odd}}n_i + 2 m_i)d_i=0$. Dividing this sum by the g.c.d. of all $d_i$'s and taking the result modulo 2 (here one uses the fact $\tfrac{n}{g.c.d.(n,k)}\mid \tbinom{n}{k}$), we obtain that the coefficient $n_1$ in the presentation~\eqref{alphaid} has to be even.

\medskip

We now compute $\cs_2(x)$. Let $\Lambda=\Z e_1\oplus\ldots\oplus\Z e_{4m}$. The root lattice is given by $T^*=\{\sum a_ie_i\mid \sum a_i \text{ is even}\}$ and
\[
\omega_1=e_1,\;\omega_2=e_1+e_2,\;\omega_3=e_1+e_2+e_3,\ldots,\omega_{4m}=e_1+\ldots +e_{4m}.
\]
By \cite[\S 2]{GaZa} we have $\cs_2(x)=\sum_{i=1}^{4m} n_i \cs_2(\widehat{e^{\omega_i}})$ and $\cs_2(\widehat{e^{\omega_i}})=N(\widehat{e^{\omega_i}})q$, where 
\[
N(\sum a_je^{\lambda_j})=\tfrac{1}{2}\sum a_j\langle\lambda_j,\alpha^{\vee}\rangle^2\text{ for a fixed long root }\alpha.
\]
If we set $\alpha=2e_{4m}$, then $\langle \lambda,\alpha^{\vee}\rangle =(\lambda,e_{4m})$ and
\[
N(\widehat{e^{\omega_i}})=\tfrac{1}{2}\sum_{\lambda\in W(\omega_i)}\langle\lambda,\alpha^{\vee}\rangle^2=\tfrac{1}{2}\sum_{{\lambda\in W(\omega_i)}}(\lambda,e_{4m})^2=2^{i-1}\tbinom{4m-1}{i-1}
\]
which is even for $i\ge 2$ (here we used the fact that the Weyl group acts by permutations and sign changes on $\{e_1,\ldots,e_{4m}\}$). Since $n_1$ is even, we get that $\cs_2(x)\in 2\Z q$.

\subsection{\it Half-spin and adjoint groups of type $D_{2m}$ ($m\ge 2$)}

\ 

\medskip

We first treat the half-spin group $G=\gHSpin_{8m}$. As in the $C_n$-case all even fundamental weights are in $T^*$ and all odd fundamental weights correspond to a generator of $\Lambda/T^*\simeq \Z/2\Z$. Therefore, the map $f\colon \Z[\Lambda] \to \Z[\Lambda/T^*]$ applied to the element $x=\sum_{i=1}^{2m}n_i\widehat{e^{\omega_i}}+\delta_i\widehat{e^{\omega_i}}$ gives the same equality $(\sum_{i \text{ is odd}}n_i + 2 m_i)d_i=0$, where $m_i\in \Z$, $d_i=2^i\binom{2m}{i}$ for $i\le 2m-2$ and $d_{2m-1}=2^{2m-1}$. Dividing by the g.c.d. of $d_i$'s and taking modulo $2$ we obtain that $n_1$ is even if $m>2$ and $n_1+n_3$ is even if $m=2$.

\medskip

We now compute $\cs_2(x)$. Take a long root $\alpha=e_{2m-1}+e_{2m}$. Then $(\alpha,\alpha)=2$ and $\langle \lambda,\alpha^{\vee}\rangle=(\lambda,e_{2m-1})+(\lambda,e_{2m})$. For $i\le 2m-2$ we have 
\[
N(\widehat{e^{\omega_i}})=\sum_{\lambda\in
W(\omega_i)}(\lambda,e_{2m})^2+(\lambda,e_{2m})(\lambda,e_{2m-1})=\sum_{\lambda\in
W(\omega_i)}(\lambda,e_{2m})^2=
2^i\tbinom{2m-1}{i-1},
\]
and $N(\widehat{e^{\omega_{2m-1}}})=N(\widehat{e^{\omega_{2m}}})=2^{2m-3}$ (here we used the fact that $W$ acts by permutations and even sign changes).

\medskip

Finally, if $m>2$, we obtain $\cs_2(x)=\sum_{i}n_i N(\widehat{e^{\omega_i}})q \in 4\Z q$, where $4\Z q = \Dec(\gHSpin_{4m})$ by~\cite[\S5]{BR}. If $m=2$, then $N(\widehat{e^{\omega_{4}}})=2$, hence, $\cs_2(x) \in 2\Z q$, where $2\Z q=\Dec(\gHSpin_{8})$ again by~\cite[\S5]{BR}.

\medskip

If $m>2$, for the adjoint group $G=\gPGO_{8m}$ by \cite[\S4]{Merkurjev} and the respective half-spin case we obtain
\[
4\Z q=\Dec(\gPGO_{8m})\subseteq \SDec(\gPGO_{8m})\subseteq \SDec(\gHSpin_{8m})=4\Z q.
\]

If $G=\gPGO_8$, direct computations (see \cite{Ne}) show that $\Dec(G)=\SDec(G)$.

\section{Applications}

Observe that $H^3(F,\Z/n\Z(2))$ is the $n$-th torsion part of $H^3(F,2)$ for every $n$ and $H^3(F,\Z/n\Z(2))= H^3(F,\mu_n^{\otimes 2})$
if $\charac(F)$ does not divide $n$.

\subsection{\it Type $C_n$}\label{typec}

Let $G=\gPGSp_{2n}$ be the split projective symplectic group. For a field extension $L/F$, the set $H^1(L,G)$ is identified with the set of isomorphism classes of central simple $L$-algebras $A$ of degree $2n$ with a symplectic involution $\sigma$ (see \cite[\S 29]{Book}). A decomposable invariant of $G$ takes an algebra with involution $(A,\sigma)$ to the cup-product $\phi \cup [A]$ for a fixed element $\phi\in F^\times$. In particular, decomposable invariants of $G$ are independent of the involution.

\medskip

Suppose that $4\mid n$. It is shown in \cite[Theorem 4.6]{Merkurjev} that the group of indecomposable invariants $\Inv^3(G,2)_\ind$ is cyclic of order $2$. If $\charac(F)\neq 2$, Garibaldi, Parimala and Tignol constructed in \cite[Theorem A]{GPT09} a degree $3$ cohomological invariant $\Delta_{2n}$ of the group $G$ with coefficients in $\Z/2\Z$. They showed that if $a\in A$ is a $\sigma$-symmetric element of $A^\times$ and $\sigma'=\operatorname{Int}(a)\circ\sigma$, then
\begin{equation}\label{conj}
\Delta_{2n}(A,\sigma')=\Delta_{2n}(A,\sigma)+\operatorname{Nrp}(a)\cup[A],
\end{equation}
where $\operatorname{Nrp}$ is the pfaffian norm. In particular, $\Delta_{2n}$ does depend on the involution and therefore, the invariant $\Delta_{2n}$ is not decomposable. Hence the the class of $\Delta_{2n}$ in $\Inv^3(G,2)_\ind$ is nontrivial.

\medskip

It follows from (\ref{conj}) that the class $\Delta_{2n}(A)\in \tfrac{H^3(L, \Z/2\Z)}{L^\times\cup [A]}$
of $\Delta_{2n}(A,\sigma)$ depends only on the $L$-algebra $A$ of degree $2n$ and exponent $2$ but not on the involution. Since $\Delta_{2n}(A,\sigma)$ is not decomposable, it is not semi-decomposable by our main theorem. The latter implies that $\Delta_{2n}(A)$ is {\em nontrivial generically}, i.e. there is a central simple algebra $A$ of degree $2n$ over a field extension of $F$ with exponent $2$ such that $\Delta_{2n}(A)\neq 0$. This answers a question raised in \cite{GPT09}. (See \cite[Remark 4.10]{Demba13} for the case $n=4$.)

\subsection{\it Type $A_{n-1}$} Let $G=\gSL_n/\gmu_m$, where $n$ and $m$ are positive integers such that $n$ and $m$ have the same prime divisors and $m\mid n$.
Given a field extension $L/F$ the natural surjection $G\to \gPGL_{n}$ yields a map
\[
\alpha:H^1(L,G)\to H^1(L,\gPGL_{n})\subset \operatorname{Br}(L)
\]
taking a $G$-torsor $Y$ over $L$ to the class of a central simple algebra $A(Y)$ of degree $n$ and exponent dividing $m$. By definition, a decomposable invariant of $G$ is of the form $Y\mapsto \phi \cup [A(Y)]$ for a fixed $\phi \in F^\times$.

\medskip

The map $\gSL_m\to  \gSL_n$ taking a matrix $M$ to the tensor product $M\otimes I_{n/m}$ with the identity matrix, gives rise to a group homomorphism
$\gPGL_m\to G$.
The induced homomorphism (see \cite[Theorem 4.4]{Merkurjev})
\[
\varphi:\Inv^3(G,2)_\norm\to \Inv^3(\gPGL_m,2)_\norm=F^\times/F^{\times m}
\]
is a splitting of the inclusion homomorphism
\[
F^\times/F^{\times m}=\Inv^3(G,2)_\dec\hookrightarrow \Inv^3(G,2)_\norm.
\]

Collecting descriptions of $p$-primary components of $\Inv^3(G,2)_\ind$ (see~\ref{nonadjA}) we get
\begin{equation}\label{an}
\Inv^3(G,2)_\ind\simeq \tfrac{m}{k}\Z q/m\Z q,\quad\text{ where }
k=\left\{
    \begin{array}{ll}
      \gcd(\frac{n}{m},m), & \hbox{if $\frac{n}{m}$ is odd;} \\
      \gcd(\frac{n}{2m},m), & \hbox{if $\frac{n}{m}$ is even.}
    \end{array}
  \right.
\end{equation}

Let $\Delta_{n,m}$ be a (unique) invariant in $\Inv^3(G,2)_\norm$ such that its class in $\Inv^3(G,2)_\ind$ corresponds to $\tfrac{m}{k}q+m\Z q$ and $\varphi(\Delta_{n,m})=0$. Note that the order of $\Delta_{n,m}$ in $\Inv^3(G,2)_\norm$ is equal to $k$. Therefore, $\Delta_{n,m}$ takes values in $H^3(-,\Z/k\Z(2))\subset H^3(-,2)$.

Fix a $G$-torsor $Y$ over $F$ and consider the twists $^Y\! G$ and $\gSL_1(A(Y))$ by $Y$ of the groups $G$ and $\gSL_{n}$ respectively.
The group $F^\times$ acts transitively on the fiber over $A(Y)$ of the map $\alpha$. If $\phi\in F^\times$, we write $^\phi Y$ for the corresponding element in the fiber. By~\eqref{an} the image of $\Delta_{n,m}$ under the natural composition
\[
\Inv^3(G,2)_\norm\simeq \Inv^3(^Y\! G,2)_\norm\longrightarrow \Inv^3(\gSL_1(A(Y)),2)_\norm
\]
is a $\tfrac{m}{k}$-multiple of the Rost invariant. Recall that the
Rost invariant takes the class of $\phi$ in
$F^\times/\Nrd(A(Y)^\times)=H^1(F,\gSL_1(A(Y)))$ to the cup-product
$\phi\cup [A(Y)]\in H^3(F,2)$. So we get
\begin{equation}\label{diff}
\Delta_{n,m}(^\phi Y)-\Delta_{n,m}(Y)\in F^\times\cup \tfrac{m}{k}[A(Y)].
\end{equation}

\medskip

Given a central simple $L$-algebra $A$ of degree $n$ and exponent dividing $m$, we define an element
\[
\Delta_{n,m}(A)\in \tfrac{H^3(L,\Z/k\Z(2))}{L^\times\cup \tfrac{m}{k}[A]}
\]
as follows. Choose a $G$-torsor $Y$ over $L$ with $A(Y)\simeq A$ and set $\Delta_{n,m}(A)$ to be the class of $\Delta_{n,m}(Y)$ in the factor group. It follows from \eqref{diff} that $\Delta_{n,m}(A)$ is independent of the choice of $Y$.

\begin{prop}
Let $A$ be a central simple $L$-algebra of degree $n$ and exponent dividing $m$. Then the order of $\Delta_{n,m}(A)$ divides $k$.
If $A$ is a generic algebra, then the order of $\Delta_{n,m}(A)$ is equal to $k$.
\end{prop}

\begin{proof}
If $k'$ is a proper divisor of $k$, then the multiple $k'\Delta_{n,m}$ is not decomposable.
By our theorem $k'\Delta_{n,m}$ is not semi-decomposable and, hence, $k'\Delta_{n,m}(A)\neq 0$.
\end{proof}

\begin{ex}
Let $A$ be a central simple $F$-algebra of degree $2n$ divisible by $8$ and exponent $2$. Choose a symplectic involution $\sigma$ on $A$.
The group $\gPGSp_{2n}$ is a subgroup of $\gSL_{2n}/\gmu_2$, hence, if $\charac(F)\neq 2$, the restriction of the invariant $\Delta_{2n,2}$ on $\gPGSp_{2n}$
is the invariant $\Delta_{2n}(A,\sigma)$ considered in subsection~\ref{typec}. It follows that $\Delta_{2n,2}(A)=\Delta_{2n}(A)$ in the group
$H^3(F, \Z/2\Z)/(F^\times\cup [A])$.
\end{ex}

The class $\Delta_{n,m}$ is trivial on decomposable algebras:

\begin{prop}\label{decomp}
Let $n_1, n_2,m$ be positive integers such that $m$ divides $n_1$ and $n_2$. Let $A_1$ and $A_2$ be two central simple
algebras over $F$ of degree $n_1$ and $n_2$ respectively and of exponent dividing $m$. Then $\Delta_{n_1n_2,m}(A_1\otimes_F A_2)=0$.
\end{prop}

\begin{proof}
The tensor product homomorphism $\gSL_{n_1} \times \gSL_{n_2} \to \gSL_{n_1n_2}$
yields a homomorphism
\[
\Sym^2(T_{n_1n_2}^*)\to \Sym^2(T_{n_1}^*)\oplus \Sym^2(T_{n_2}^*),
\]
where $T_{n_1}$, $T_{n_2}$ and $T_{n_1n_2}$ are maximal tori of respective groups. The image of the canonical Weyl-invariant generator $q_{n_1n_2}$
of $\Sym^2(T_{n_1n_2}^*)$ is equal to $n_2 q_{n_1} + n_1 q_{n_2}$. Since $n_1$ and $n_2$ are divisible by $m$, the pull-back of the invariant
$\Delta_{n_1n_2,m}$ under the homomorphism $(\gSL_{n_1}/\gmu_m) \times (\gSL_{n_2}/\gmu_m) \to \gSL_{n_1n_2}/\gmu_m$
is trivial.
\end{proof}

\section{Appendix}

The aim of this section is to verify that the groups of decomposable and semi-decomposable cohomological invariants of the group $\gPGO_8$ coincide.
Following the notation of~\cite{Merkurjev} we have 
\begin{itemize}
\item $\Lambda=\Z e_1\oplus\ldots\Z e_4+\Z e$ where $e=\frac{1}{2}(e_1+e_2+e_3+e_4)$,
\item $T^*$ consists of all $\sum a_ie_i$ with $\sum a_i$ even.
\item $S^2(\Lambda)^W=\Z q$ where $q=\frac{1}{2}(e_1^2+e_2^2+e_3^2+e_4^2)$ and
\item Fundamental weights are 
\[\omega_1=e_1, \omega_2=e_1+e_2, \omega_3=e-e_4,\omega_4=e.\]
\item Simple roots are 
\[\lambda_1=e_1-e_2,\lambda_2=e_2-e_3,\lambda_3=e_3-e_4,\lambda_4=e_3+e_4.\]
\item The Weyl group $W=S_4\rightthreetimes (C_2)^3$ consists of permutations of $e_i$ and sign changes of even number of variables.
\item The $W$-orbits of fundamental weights are given by:

 $W(\omega_1)=\{e_1,e_2,e_3,e_4,-e_1,-e_2,-e_3,-e_4\}$,

 $W(\omega_2)=\{e_1+e_2,\ldots,e_3+e_4,-(e_1+e_2),\ldots,-(e_3+e_4),e_1-e_2,\ldots e_3-e_4,-e_1+e_2,\ldots,-e_3+e_4\}$,

 $W(\omega_3)=\{e-e_1,e-e_2,e-e_3,e-e_4\}$ and
 
$W(\omega_4)=\{e,-e,e-e_1-e_2,e-e_1-e_3,e-e_1-e_4,e-e_2-e_3,e-e_2-e_4,e-e_3-e_4\}$.
\end{itemize}

\medskip

Let $\widehat{e^{\omega_i}}$ denote the sum $\sum_{\lambda\in W(\omega_i)}e^{\lambda}-1$.
Then the ideal $\widetilde{I}^W$ is generated by $\widehat{e^{\omega_i}}, i=1,\ldots, 4$
Note that $\Z[\Lambda]$ is the Laurent polynomial ring $\Z[e^{\pm e_1},\ldots e^{\pm e_4},e^{\pm e}]$ so we represetn it as a quotient of the polynomial ring 
\[Z[\Lambda]=\Z[u_1,v_1,\ldots u_4,v_4,u_5,v_5]/(u_iv_1-1,\ldots u_5v_5-1, u_5^2=u_1u_2u_3u_4)\]
where $u_i= e^{e_i}$, $v_i= e^{-e_i}$ for $i=1\ldots 4$ and $u_5= e^e$ and $v_5=e^{-e}$.
Let $r_i=e^{\lambda_i}$ and $s_i=e^{-\lambda_i}$ for the simple roots $\lambda_i$.

We have an exact sequence
\[
0\to J_1\to \Z [u_1,\ldots, u_5,v_1,\ldots, v_5,r_1,\ldots, r_4,s_1,\ldots ,s_4]\stackrel{\pi}\to\Z[\Lambda]\to 0,\]
where
\[J_1=(u_1v_1-1,\ldots, u_5v_5-1,r_1s_1-1,\ldots, r_4s_4-1,
\]
\[ r_1-u_1v_2,r_2-u_2v_3,r_3-u_3v_4,r_4-u_3u_4,
 u_1u_2u_3u_4-u_5^2)\]
  and the preimage 
\[
\pi^{-1}(\widetilde{I}^W)=J_1+(u_1+\ldots u_4+v_1+\ldots v_4-8,
\]
\[
 u_1u_2+\ldots u_3u_4+v_1v_2+\ldots v_3v_4+\ldots u_1v_2+\ldots u_3v_4+v_1u_2+\ldots v_3u_4-24, 
 \]
 \[
 u_5v_1+u_5v_2+u_5v_3+u_5v_4+u_1v_5+u_2v_5+u_3v_5+u_4v_5-8,
 \]
 \[
 u_5+v_5+u_5v_1v_2+u_5v_1v_3+u_5v_1v_4+u_5v_2v_3+u_5v_2v_4+u_5v_3v_4-8).
 \]
Note that $\Z[T^*]=\pi(\Z[r_1,\ldots r_4,s_1,\ldots s_4])$ and \[\widetilde{I}^W\cap\Z[T^*]=\pi(\pi^{-1}(\widetilde{I}^W)\cap\Z[r_i,s_i])\]
Now we use the Maple software to compute a basis of a bigger intersection: \[\pi^{-1}(\widetilde{I}^W+\widetilde{I}^4)\cap\Z[r_i,s_i].\]
We will prove that for any $x$ in the basis of this intersection~\cite[\S 4b,p.19]{Merkurjev}\[
c_2(\pi(x))\in 4q\Z=\Dec(G).\]
Since $c_2(\widetilde{I}^3)=0$, it is enough to consider the generators that are not contained in $\pi^{-1}(\widetilde{I}^3)$.
We compute the basis of the intersection using the following code:

{
\tt
with(PolynomialIdeals):\\
\# relations ideal:
J1 := $\langle u_1v_1-1, u_2v_2-1, u_3v_3-1, u_4v_4-1, u_5v_5-1, r_1s_1-1, r_2s_2-1, r_3s_3-1, r_4s_4-1, u_1u_2u_3u_4-u_5^2, r_1-u_1v_2, r_2-u_2v_3, r_3-u_3v_4, r_4-u_3u_4\rangle$\\
\# preimage of $\widetilde{I}^W$:
J2 := $\langle u_1+u_2+u_3+u_4+v_1+v_2+v_3+v_4-8, u_5v_1+u_5v_2+u_5v_3+u_5v_4+u_1v_5+u_2v_5+u_3v_5+u_4v_5-8, u_5+v_5+u_5v_1v_2+u_5v_1v_3+u_5v_1v_4+u_5v_2v_3+u_5v_2v_4+u_5v_3v_4-8, -24+u_1v_2+u_2v_3+u_3v_4+u_3u_4+u_1u_2+u_1v_3+u_1v_4+u_2v_4+v_1u_2+u_1u_3+u_1u_4+u_2u_3+u_2u_4+v_1v_2+v_1v_3+v_1v_4+v_2v_3+v_2v_4+v_3v_4+v_1u_3+v_1u_4+v_2u_3+v_2u_4+v_3u_4\rangle$\\
\# preimage of the augmentation ideal:\\
augL := $\langle u[1]-1, v[1]-1, u[2]-1, v[2]-1, u[3]-1, v[3]-1, u[4]-1, v[4]-1, u[5]-1, v[5]-1\rangle$;\\
\# preimages of the square,cube and fourth power of the augmentation ideal:
squarL := Add(Multiply(augL, augL), J1);\\
cubL := Add(Multiply(augL, Multiply(augL, augL)), J1);\\ 
quadL := Add(Multiply(augL, cubL), J1); \\
\# preimage of $\widetilde{I}^W+\widetilde{I}^4$:\\
J := Add(Add(J1, J2), quadL)\\
\# intersection with the subring $\Z[r_i,s_i]$:\\
K := EliminationIdeal(J, {r[1], r[2], r[3], r[4], s[1], s[2], s[3], s[4]}):\\
\# basis of the intersection\\
Gen := IdealInfo[Generators](K):\\
\# print out the elements of the basis that do not lie in $\pi^{-1}(\widetilde{I}^3)$\\
for x in Gen do \\
if not(IdealMembership(x, cubL)) then print(x) end if \\
end do\\
}
This gives a list of 18 polynomials:
\begin{itemize}
\item $ -34 - r_1 s_4 - 2 r_2 s_4 + 6 r_1 + 6 s_1 + 10 r_2
+ 10 s_2 + 4 r_3 + 4 s_3 + 4 r_4 + 4 s_4 - 2 s_2 r_4
- s_1 r_4 - 2 s_2 r_3 - s_1 r_3 + r_4 r_3
- 2 s_3 r_2 - 2 s_1 r_2 - s_3 r_1 - 2 s_2 r_1
+ s_3 s_4$;

\item $   38 + r_1 s_4 + 2 r_2 s_4 + r_3 s_4 - 6 r_1 - 6 s_1
- 10 r_2 - 10 s_2 - 6 r_3 - 6 s_3 - 6 r_4 - 6 s_4
      + s_3 r_4 + 2 s_2 r_4 + s_1 r_4 + 2 s_2 r_3
      + s_1 r_3 + 2 s_3 r_2 + 2 s_1 r_2 + s_3 r_1
      + 2 s_2 r_1$;

 \item $-37 - r_1 s_4 - 2 r_2 s_4 - 3 r_3 s_4 + 6 r_1 + 6 s_1
    + 10 r_2 + 10 s_2 + 7 r_3 + 5 s_3 + 4 r_4 + 7 s_4
  - 2 s_2 r_4 - s_1 r_4 - 2 s_2 r_3 - r^2_3  - s_1 r_3
   + r_4 r_3 - 2 s_3 r_2 - 2 s_1 r_2 - s_3 r_1
  - 2 s_2 r_1 + r^2_3  s_4$;
	
 \item $35 + r_1 s_4 + 2 r_2 s_4 + r_3 s_4 - 6 r_1 - 6 s_1
    - 10 r_2 - 10 s_2 - 3 r_3 - 5 s_3 - 3 r_4 - 6 s_4
    + 2 s_2 r_4 + s_1 r_4 + 2 s_2 r_3 - r^2_3  + s_1 r_3
    - 3 r_4 r_3 + 2 s_3 r_2 + 2 s_1 r_2 + s_3 r_1
    + 2 s_2 r_1 + r^2_3  r_4$;
		
\item $-118 - 6 r_2 s_4 + 14 r_1 + 9 s_1 + 74 r_2 + 46 s_2
   + 14 r_3 + 9 s_3 + 14 r_4 + 9 s_4 + r^2_4  - 10 s_2 r_4
   - 10 s_2 r_3 + r^2_3  + r_4 r_3 - 6 s_3 r_2 - 6 s_1 r_2
   - 8 r_4 r_2 - 8 r_3 r_2 + r^2_1  - 8 r^2_2  - 10 s_2 r_1
   + r_4 r_1 + r_3 r_1 - 8 r_2 r_1 + 3 r_1 r_3 r_4$;
	
\item $-92 + 28 r_1 + 18 s_1 + 52 r_2 + 26 s_2 + 34 r_3 + 12 s_3
   - 8 r_4 + 2 r^2_4  - 2 s_2 r_4 - 8 s_2 r_3 - r^2_3 
   - 9 s_1 r_3 + 2 r_4 r_3 - 6 s_3 r_2 - 12 s_1 r_2
   + 2 r_4 r_2 - 16 r_3 r_2 - r^2_1  - 4 r^2_2  - 3 s_3 r_1
   - 8 s_2 r_1 + 2 r_4 r_1 - 4 r_3 r_1 - 10 r_2 r_1
   + 6 r_2 r_3 s_1$;
	
\item $-92 + 34 r_1 + 12 s_1 + 46 r_2 + 32 s_2 + 34 r_3 + 12 s_3
   - 8 r_4 + 2 r^2_4  - 2 s_2 r_4 - 14 s_2 r_3 - r^2_3 
   - 3 s_1 r_3 + 2 r_4 r_3 - 6 s_3 r_2 - 6 s_1 r_2
   + 2 r_4 r_2 - 10 r_3 r_2 - r^2_1  - 4 r^2_2  - 3 s_3 r_1
   - 14 s_2 r_1 + 2 r_4 r_1 - 10 r_3 r_1 - 10 r_2 r_1
   + 6 r_1 r_3 s_2$;
	
\item $-92 + 34 r_1 + 12 s_1 + 52 r_2 + 26 s_2 + 28 r_3 + 18 s_3
   - 8 r_4 + 2 r^2_4  - 2 s_2 r_4 - 8 s_2 r_3 - r^2_3 
   - 3 s_1 r_3 + 2 r_4 r_3 - 12 s_3 r_2 - 6 s_1 r_2
   + 2 r_4 r_2 - 10 r_3 r_2 - r^2_1  - 4 r^2_2  - 9 s_3 r_1
   - 8 s_2 r_1 + 2 r_4 r_1 - 4 r_3 r_1 - 16 r_2 r_1
   + 6 r_1 r_2 s_3$;
	
\item $-92 - 9 r_1 s_4 - 12 r_2 s_4 + 34 r_1 + 12 s_1 + 52 r_2
   + 26 s_2 - 8 r_3 + 28 r_4 + 18 s_4 - r^2_4  - 8 s_2 r_4
   - 3 s_1 r_4 - 2 s_2 r_3 + 2 r^2_3  + 2 r_4 r_3
   - 6 s_1 r_2 - 10 r_4 r_2 + 2 r_3 r_2 - r_1^2  - 4 r^2_2 
  - 8 s_2 r_1 - 4 r_4 r_1 + 2 r_3 r_1 - 16 r_2 r_1
   + 6 r_1 r_2 s_4$;
	
\item $-92 - 3 r_1 s_4 - 6 r_2 s_4 + 28 r_1 + 18 s_1 + 52 r_2
   + 26 s_2 - 8 r_3 + 34 r_4 + 12 s_4 - r^2_4  - 8 s_2 r_4
   - 9 s_1 r_4 - 2 s_2 r_3 + 2 r^2_3  + 2 r_4 r_3
- 12 s_1 r_2 - 16 r_4 r_2 + 2 r_3 r_2 - r^2_1  - 4 r^2_2 
   - 8 s_2 r_1 - 4 r_4 r_1 + 2 r_3 r_1 - 10 r_2 r_1
   + 6 r_2 r_4 s_1$;
	
\item $-92 - 3 r_1 s_4 - 6 r_2 s_4 + 34 r_1 + 12 s_1 + 46 r_2
   + 32 s_2 - 8 r_3 + 34 r_4 + 12 s_4 - r^2_4  - 14 s_2 r_4
   - 3 s_1 r_4 - 2 s_2 r_3 + 2 r^2_3  + 2 r_4 r_3
   - 6 s_1 r_2 - 10 r_4 r_2 + 2 r_3 r_2 - r^2_1  - 4 r^2_2 
   - 14 s_2 r_1 - 10 r_4 r_1 + 2 r_3 r_1 - 10 r_2 r_1
   + 6 r_1 r_4 s_2$;
	
 \item $80 - 22 r_1 - 12 s_1 - 40 r_2 - 26 s_2 - 22 r_3 - 12 s_3
    + 8 r_4 - 2 r^2_4  + 2 s_2 r_4 + 8 s_2 r_3 + r^2_3 
    + 3 s_1 r_3 - 2 r_4 r_3 + 6 s_3 r_2 + 6 s_1 r_2
    - 2 r_4 r_2 + 4 r_3 r_2 + r^2_1  + 4 r^2_2  + 3 s_3 r_1
    + 8 s_2 r_1 - 2 r_4 r_1 - 2 r_3 r_1 + 4 r_2 r_1
    + 6 r_1 r_2 r_3$;
		
\item $80 + 3 r_1 s_4 + 6 r_2 s_4 - 22 r_1 - 12 s_1 - 40 r_2
   - 26 s_2 + 8 r_3 - 22 r_4 - 12 s_4 + r^2_4  + 8 s_2 r_4
   + 3 s_1 r_4 + 2 s_2 r_3 - 2 r^2_3  - 2 r_4 r_3
   + 6 s_1 r_2 + 4 r_4 r_2 - 2 r_3 r_2 + r^2_1  + 4 r^2_2 
   + 8 s_2 r_1 - 2 r_4 r_1 - 2 r_3 r_1 + 4 r_2 r_1
   + 6 r_1 r_2 r_4$;
	
\item $-34 - 3 r_1 s_4 + 26 r_1 + 18 s_1 - 10 r_2 + 4 s_2 - 4 r_3
   + 6 s_3 - 4 r_4 + 6 s_4 + r^2_4  + 2 s_2 r_4 - 3 s_1 r_4
   + 2 s_2 r_3 + r^2_3  - 3 s_1 r_3 - 2 r_4 r_3
   - 6 s_1 r_2 + 4 r_4 r_2 + 4 r_3 r_2 - 2 r^2_1  + 4 r^2_2 
   - 3 s_3 r_1 - 4 s_2 r_1 - 2 r_4 r_1 - 2 r_3 r_1
   - 2 r_2 r_1 + 6 r_2 r_3 r_4$;
	
\item $22 + 3 r_1 s_4 - 26 r_1 - 18 s_1 + 16 r_2 + 2 s_2 + 16 r_3
   - 6 s_3 + 16 r_4 - 6 s_4 - r^2_4  - 8 s_2 r_4
   + 3 s_1 r_4 - 8 s_2 r_3 - r^2_3  + 3 s_1 r_3
   - 10 r_4 r_3 + 6 s_1 r_2 - 10 r_4 r_2 - 10 r_3 r_2
   + 2 r^2_1  - 4 r^2_2  + 3 s_3 r_1 + 4 s_2 r_1 + 2 r_4 r_1
   + 2 r_3 r_1 + 2 r_2 r_1 + 6 r_3 r_4 s_2$;
	
\item $112 - 3 r_1 s_4 + 6 r_2 s_4 - 3 r_3 s_4 - 8 r_1 - 9 s_1
   - 74 r_2 - 46 s_2 - 8 r_3 - 9 s_3 - 11 r_4 - 6 s_4
   - r^2_4  + 10 s_2 r_4 + 10 s_2 r_3 - r^2_3  - 4 r_4 r_3
   + 6 s_3 r_2 + 6 s_1 r_2 + 8 r_4 r_2 + 8 r_3 r_2
   - r^2_1  + 8 r^2_2  + 10 s_2 r_1 - 4 r_4 r_1 - 7 r_3 r_1
   + 8 r_2 r_1 + 3 r_1 r_3 s_4$;
	
\item $112 + 6 r_2 s_4 - 11 r_1 - 6 s_1 - 74 r_2 - 46 s_2 - 8 r_3
   - 9 s_3 - 8 r_4 - 9 s_4 - r^2_4  + 10 s_2 r_4
   - 3 s_1 r_4 + 10 s_2 r_3 - r^2_3  - 3 s_1 r_3
   - 7 r_4 r_3 + 6 s_3 r_2 + 6 s_1 r_2 + 8 r_4 r_2
   + 8 r_3 r_2 - r^2_1  + 8 r^2_2  + 10 s_2 r_1 - 4 r_4 r_1
   - 4 r_3 r_1 + 8 r_2 r_1 + 3 r_3 r_4 s_1$;
	
\item $22 + 3 r_1 s_4 - 6 r_2 s_4 - 6 r_3 s_4 - 26 r_1 - 18 s_1
   + 22 r_2 - 4 s_2 + 16 r_3 - 6 s_3 + 10 r_4 - r^2_4 
   - 2 s_2 r_4 + 3 s_1 r_4 - 2 s_2 r_3 - r^2_3 
   + 3 s_1 r_3 - 4 r_4 r_3 + 6 s_1 r_2 - 10 r_4 r_2
   - 16 r_3 r_2 + 2 r^2_1  - 4 r^2_2  + 3 s_3 r_1 + 4 s_2 r_1
   + 2 r_4 r_1 + 2 r_3 r_1 + 2 r_2 r_1 + 6 r_2 r_3 s_4$;
\end{itemize}

Take the first element of the list
\[y=-34 - r_1 s_4 - 2 r_2 s_4 + 6 r_1 + 6 s_1 + 10 r_2
+ 10 s_2 + 4 r_3 + 4 s_3 + 4 r_4 + \]
\[
4 s_4 - 2 s_2 r_4
- s_1 r_4 - 2 s_2 r_3 - s_1 r_3 + r_4 r_3
- 2 s_3 r_2 - 2 s_1 r_2 - s_3 r_1 - 2 s_2 r_1
+ s_3 s_4.\]

We compute $c_2(y)$ as the second term in the power series expansion of
\[
(1+(l_1-l_4)t)^{-1}(1+(l_2+l_4)t)^{-2}(1+l_1t)^6(1-l_1t)^6(1+l_2t)^{10}\]
\[
(1-l_2t)^{10}(1+l_3t)^4(1-l_3t)^4(1+l_4t)^4(1-l_4t)^4(1+(l_2+l_4)t)^{-2}\]
\[
(1+(-l_1+l_4)t)^{-1}(1+(-l_2+l_3)t)^{-2}(1+(-l_1+l_3)t)^{-1}(1+(l_4+l_3)t)\]
\[
(1+(-l_3+l_2)t)^{-2}(1+(-l_1+l_2)t)^{-2}(1+(-l_3+l_1)t)^{-1}(1+(-l_2+l_1)t)^{-2}
(1+(-l_3+l_4)t)
\]
where $l_1 = e1]-e_2, l_2 = e_2-e_3, l_3 = e_3-e_4, l_4 = e_3+e_4$.

Computation shows that $c_2(y)=-2(e_1^2+e_2^2+e_3^2+e_4^2)=-4q$.

As the last step we show that for every generator $x$ that does not lie in $\pi^{-1}(\widetilde{I}^3)$ we have that either $x-y$ or $x+y$ $x-2y$ or $x+2y$ lies in $\pi^{-1}(\widetilde{I}^3)$.

To do this we use the following Maple code:

{
\tt
for x in Gen do \\
if not IdealMembership(x, cubL) and IdealMembership(x, squarL)
and \\ (IdealMembership(x-y, cubL) or IdealMembership(x+y, cubL) or \\ IdealMembership(x+2y, cubL) or IdealMembership(x-2y, cubL) )  \\ then print(x) end if \\
end do
}

It returns the same list of 18 polynmials, so we see that for every generator $x$ we have $c_2(x)\in 4q\Z,$ so $\SDec(G)\subseteq\Dec(G).$

\end{document}